\documentclass{amsart}
\usepackage[demo]{graphicx}
\usepackage{caption}
\usepackage{subcaption}
\usepackage{graphicx, amsmath, amsfonts, epsfig, latexsym, amssymb, amsthm}
\usepackage{tikz}
\usepackage{tkz-berge}
\usepackage{float}

\usetikzlibrary {positioning}
\definecolor {processblue}{cmyk}{0.96,0,0,0}
\newtheorem{thm}{Theorem}[section]
\theoremstyle{definition}
\newtheorem{cor}[thm]{Corollary}
\newtheorem{prop}[thm]{Proposition}

\newtheorem{lem}[thm]{Lemma}

\newtheorem{notes}[thm]{Notation}
\newtheorem{rem}[thm]{Remark}
\newtheorem{ex}[thm]{Example}

\numberwithin{equation}{section}
\begin{document}
\title[Some results on the ideal-based cozero-divisor graph of a commutative ring]
{Some results on the ideal-based cozero-divisor graph of a commutative ring}
\author{ F. Farshadifar}
\address{Department of Mathematics Education, Farhangian University, P.O. Box 14665-889, Tehran, Iran.}
\email{f.farshadifar@cfu.ac.ir}


\begin{abstract}
Let $R$ be a commutative ring with
identity and $I$ be an ideal of $R$.
The cozero-divisor graph with respect to $I$, denoted by $\Gamma''_I(R)$, is the
graph of $R$ with vertices $\{x \in R\setminus I\: |\: xR+I\not=R \}$ and two distinct vertices $x$ and $y$ are adjacent if and only if  $x \not \in yR+I$ and $y \not \in xR+I$.
In this paper, we obtained some results on $\Gamma''_I(R)$.
\end{abstract}
\maketitle
\section{Introduction}
\noindent
Throughout this paper, we denote $R$ as a commutative ring with identity and $\Bbb Z_n$ as the ring of integers modulo $n$. The set of maximal ideals and Jacobson radical of $R$ are denoted by $Max(R)$ and $J(R)$, respectively.

A \emph{graph} $G$ is defined as the pair $(V(G),E(G))$, where $V(G)$ is the set of vertices of $G$ and $E(G)$ is the set of edges of $G$. For two distinct vertices $a$ and $b$ of $V(G)$, the notation $a\sim b$ means that $a$ and $b$ are adjacent. A graph $G$ is said to be \emph{complete} if $a \sim b$ for all distinct $a, b\in V(G)$, and $G$ is said to be \emph{empty} if $E(G) =\emptyset$. Note that by this definition that a graph may be empty even if $V (G)\not =\emptyset$. An empty graph could also be described as totally disconnected. If $|V (G)|\geq 2$, a \emph{path} from $a$ to $b$ is a series
of adjacent vertices $a\sim v_1\sim v_2\sim ...\sim v_n\sim b$. The \emph{length of a path} is the number of edges it contains. A \emph{cycle} is a path that begins and ends at the same vertex in which no edge is repeated, and all vertices other than the starting and ending vertex are distinct. If a graph $G$ has a cycle, the \emph{girth} of $G$ (notated $g(G)$) is defined as the length of the shortest cycle of $G$; otherwise, $g(G) =\infty$. A graph $G$ is \emph{connected} if for every pair of distinct vertices $a, b\in V (G)$, there exists a path from $a$ to $b$. If there is a path from $a$ to $b$ with $a, b \in V (G)$, then the \emph{distance from} $a$ to $b$ is the length of the shortest path from $a$ to $b$ and is denoted by $d(a, b)$. If there is not a path between $a$ and $b$, $d(a, b) = \infty$. The \emph{diameter} of $G$ is diam$(G) = \sup\{d(a,b) \mid a, b \in V(G)\}$.
A graph $G$ is said to be \textit{planar} if it can be drawn in the plane so that its edges intersect only at their ends.
A subdivision of a graph is any graph that can be obtained from the original graph by replacing edges by paths.
By Kuratowski's Theorem, a graph $G$ is planar if and only if it does not contain a subgraph which is a subdivision of $K^5$ or $K^{3,3}$, where $K^n$ is a complete graph with $n$ vertices and $K^{m,n}$ is a complete bipartite graph, for positive integers $m, n$ \cite[p. 153]{BM76}.

Let $Z(R)$ be the set of all zero-divisors of $R$. Anderson and Livingston, in \cite{2}, introduced the \emph{zero-divisor graph of R}, denoted by $\Gamma(R)$, as the (undirected) graph with vertices $Z^*(R) = Z(R)\backslash \{0\}$ and for two distinct elements $x$ and $y$ in $Z^*(R)$, the vertices $x$ and $y$ are adjacent if and only if $xy = 0$.
In \cite{3}, Redmond introduced the definition of the zero-divisor graph with respect to an ideal. Let $I$ be an ideal of $R$. The \emph{zero-divisor graph of $R$ with respect to $I$}, denoted by $\Gamma_I(R)$, is the graph whose vertices are the set $\{ x \in R\setminus I\, |\, xy \in I\ for\ some\ y \in R \setminus I\}$ with distinct vertices $x$ and $y$ are adjacent if and only if $xy \in I$. Thus if $I = 0$, then $\Gamma_I(R) = \Gamma(R)$.

In \cite{1}, Afkhami and Khashayarmanesh introduced the \emph{cozero-divisor graph} $\Gamma'(R)$ of $R$, in which the vertices are precisely the non-zero, non-unit elements of $R$, denoted by $W^*(R)$, and two distinct vertices $x$ and $y$ are adjacent if and only if $x \not \in yR$ and $y \not \in xR$. More information about this graph can be found in \cite{1111, 2222,3333}. Let $I$ be an ideal of $R$.
The authors in \cite{00}, introduced and studied a generalization of cozero-divisor graph $\acute{\Gamma}_I(R)$ of $R$ with vertices $\{x \in R  \setminus Ann_R(I)\: |\: xI \neq I \}$ and two distinct vertices $x$ and $y$ are adjacent if and only if $x \not \in yI$ and $y \not \in xI$. In fact, $\acute{\Gamma}_I(R)$ is a generalization of cozero-divisor graph introduced in \cite{1} when $I = R$.

In \cite{FF5726}, the present author introduced and studied a generalization of cozero-divisor graph with respect to $I$, denoted by $\Gamma''_I(R)$, an undirected graph with vertices $\{x \in R\setminus I\: |\: xR+I\not=R \}$ and two distinct vertices $x$ and $y$ are adjacent if and only if  $x \not \in yR+I$ and $y \not \in xR+I$.
In fact, $\Gamma''_I(R)$ can be regarded as a dual notion of ideal-based zero-divisor graph introduced in \cite{3} and also,
$\Gamma''_I(R)$ is a generalization of cozero-divisor graph introduced in \cite{1} when $I = 0$.
The aim of this article is to obtain further results on the graph $\Gamma''_I(R)$. Similar to the cozero-divisor graph, we study the graph $\Gamma''_I(R)$ for the direct products of two commutative rings. Also, we characterize finite commutative rings $R$
such that $\Gamma''_I(R)$ is planar.
\section{$\Gamma''_I(R)$ for the direct products of two commutative rings}
\noindent
Throughout this section, all rings are commutative rings with non-zero identities.
\begin{lem}\label{2.1}
Let $R=R_1\times \cdots \times R_n$ be a direct product of commutative rings and $I=I_1\times \cdots \times I_n$ be an ideal of $R$. If $x_i$ is adjacent to $y_i$ in $\Gamma''_{I_i}(R_i)$ for some $1\leq i\leq n$, then every element in $R$ with $i$-th component $x_i$ is adjacent to all elements in $R$ with i-th component $y_i$.
\end{lem}
\begin{proof}
Assume that $x_i$ is adjacent to $y_i$ in $\Gamma''_{I_i}(R_i)$ for some $1\leq i\leq n$ and assume on the contrary that the vertices $(x_1,\ldots , x_n)$  and $(y_1,\ldots , y_n)$ are not adjacent in $\Gamma''_{I}(R)$. Without loss of generality, we can suppose that
$(x_1,\ldots , x_n)\in (y_1,\ldots , y_n)(R_1\times \cdots \times R_n)+I_1\times \cdots \times I_n$. Then
$(x_1,\ldots , x_n)=(y_1,\ldots , y_n)(r_1,\ldots , r_n)+(a_1,\ldots , a_n)$ for some $(r_1,\ldots , r_n)\in R_1\times \cdots \times R_n$ and $(a_1,\ldots , a_n) \in I_1\times \cdots \times I_n$. It follows that $x_i=y_ir_i+a_i$ and so $x_i$ is not adjacent to $y_i$ in $\Gamma''_{I_i}(R_i)$. This is a desired contradiction.
\end{proof}

\begin{cor}\label{2.2}
Let $R=R_1\times  R_2$ be a direct product of commutative rings and $I=I_1\times I_2$ be an ideal of $R$.  If $(x_1,y_1), (x_2, y_2)\in V(\Gamma''_{I_1}(R_1))\times V(\Gamma''_{I_2}(R_2))$ such that they are not adjacent in
$\Gamma''_{I_1\times I_2}(R_1\times R_2)$, then $x_1$ is not adjacent to $x_2$ in $\Gamma''_{I_1}(R_1)$ and
$y_1$ is not adjacent to $y_2$ in $\Gamma''_{I_2}(R_2)$.
\end{cor}
\begin{proof}
This follows from Lemma \ref{2.1}.
\end{proof}

\begin{lem}\label{2.3}
Let $R=R_1\times  R_2$ be a direct product of commutative rings and $I=I_1\times I_2$ be an ideal of $R$.
\begin{itemize}
\item [(a)] Let $x \in R_1$ and $y_1,y_2 \in R_2$. Then $(x,y_1)$ is adjacent to $(x,y_2)$ in $\Gamma''_{I_1\times I_2}(R_1\times R_2)$ if and only if $y_1$ is adjacent to $y_2$ in  $\Gamma''_{I_2}(R_2)$.
\item [(b)] Let $x _1, x_2 \in R_1$ and $y \in R_2$. Then $(x_1,y)$ is adjacent to $(x_2,y)$ in $\Gamma''_{I_1\times I_2}(R_1\times R_2)$ if and only if $x_1$ is adjacent to $x_2$ in  $\Gamma''_{I_1}(R_1)$.
\end{itemize}
\end{lem}
\begin{proof}
(a) Assume that $(x,y_1)$ is adjacent to $(x,y_2)$ in $\Gamma''_{I_1\times I_2}(R_1\times R_2)$.
If $y_2 \in I_2$, then we have
$$
(x,y_2)=(x,y_1)(1,0)+(0,y_2)\in (x,y_1)(R_1\times  R_2)+I_1\times I_2,
$$
which is a contradiction. Thus $y_2\not \in I_2$. Similarly, $y_1\not \in I_2$. So, if $y_2 \not \in V(\Gamma''_{I_2}(R_2))$, then $y_2R_2+I_2=R_2$ since $y_2\not \in I_2$. Thus $1_{R_2}=r_2y_2+a_2$ for some $r_2 \in R_2$ and $a_2 \in I_2$. This implies that
$$
(x,y_1)=(x,1_{R_2}y_1)=(x, (r_2y_2+a_2)y_1)=
$$
$$
(x,y_2)(1_{R_1},r_2 y_1)+(0,a_2y_2)\in (x,y_2)(R_1\times  R_2)+I_1\times I_2,
$$
This is a contradiction.
Thus $y_2 \in V(\Gamma''_{I_2}(R_2))$. Similarly, $y_1 \in V(\Gamma''_{I_2}(R_2))$. Now if $y_1$ is not adjacent to $y_2$, then $y_1 \in y_2R_2+I_2$ or $y_2 \in y_1R_2+I_2$.  Hence, without loss of generality, we can assume that $y_1=y_2r_2+a_2$ for some $r_2 \in R_2$ and $a_2 \in I_2$. This implies that $(x,y_1)=(1_{R_1}, r_2)(x,y_2)+(0, a_2)$, which means that $(x,y_1)$ is not adjacent to $(x,y_2)$. This contradiction shows that $y_1$ is adjacent to $y_2$ in  $\Gamma''_{I_2}(R_2)$. Then converse follows from Lemma \ref{2.1}.

(b) The proof is similar to the part (a).
\end{proof}

\begin{cor}\label{2.4}
Let $R=R_1\times  R_2$ be a direct product of commutative rings and $I=I_1\times I_2$ be an ideal of $R$, where $I_1$ and $I_2$ are two proper ideals of $R_1$ and $R_2$, respectively.
If either $\Gamma''_{I_1}(R_1)$ or $\Gamma''_{I_2}(R_2)$ is not planar, then $\Gamma''_{I}(R)$ is not planar.
\end{cor}
\begin{proof}
Without loss of generality, assume that $\Gamma''_{I_1}(R_1)$ is not planar. So, by Kuratowski's Theorem, it contains a subdivision of $K^5$ or $K^{3,3}$. Now, by Lemma  \ref{2.3} (b), we conclude that $\Gamma''_{I}(R)$ is not planar.
\end{proof}

The \textit{chromatic number} of a graph $G$, denoted by $\chi(G)$, is the minimal number of colors which
can be assigned to the vertices of $G$ in such a way that every two adjacent vertices have
different colors. Also, a \textit{clique of a graph} is a complete subgraph and the number of vertices
in a largest clique of $G$, denoted by $\omega(G)$, is called the \textit{clique number} of $G$.
\begin{cor}\label{2.4}
Let $R=R_1\times  R_2$ be a direct product of commutative rings and $I=I_1\times I_2$ be an ideal of $R$. Then we have the following.
\begin{itemize}
\item [(a)] $\omega (\Gamma''_{I_1\times I_2}(R_1\times R_2))\geq Max \{\omega (\Gamma''_{I_1}(R_1)), \omega(\Gamma''_{I_2}(R_2))\}$.
\item [(b)] $\chi(\Gamma''_{I_1\times I_2}(R_1\times R_2))\geq Max \{\chi(\Gamma''_{I_1}(R_1)), \chi(\Gamma''_{I_2}(R_2))\}$.
\end{itemize}
\end{cor}
\begin{proof}
This follows from Lemma  \ref{2.3}.
\end{proof}

\begin{lem}\label{2.5}
Let $R=R_1\times  R_2$ be a direct product of commutative rings and $I=I_1\times I_2$ be an ideal of $R$.
Let $x \in R_1\setminus I_1$ and $y \in R_2\setminus I_2$. Then $(x,b)$ is adjacent to $(a,y)$ in $\Gamma''_{I_1\times I_2}(R_1\times R_2)$ for each $a \in I_1$ and $b \in I_2$.
\end{lem}
\begin{proof}
Assume contrary that $(x,b)\in (a,y)(R_1\times  R_2)+I_1\times I_2$. Then $(x,b)=(a,y)(r_1,r_2)+(a_1,b_1)$, for $r_1 \in R_1, r_2 \in R_2$ and $a_1 \in I_1, b_1 \in I_2$. This implies that $x=r_1a+a_1 \in I_1$, which is a contradiction.
\end{proof}

\begin{rem}\label{1978}
Let $R=R_1\times  R_2$ be a direct product of commutative rings and $I=I_1\times I_2$ be an ideal of $R$, where $I_1$ and $I_2$ are two proper ideals of $R_1$ and $R_2$, respectively. Then
it is easy to see that $(x,y)$ is adjacent to $(a, b+u_2)$ and $(x,y)$ is adjacent to $(a+u_1, b)$ for each $a\in I_1$, $b \in I_2$, $u_1 \in U(R_1)$, and $u_2 \in U(R_2)$.
\end{rem}

\begin{prop}\label{1980}
Let $R=R_1\times  R_2$ be a direct product of commutative rings and $I=I_1\times I_2$ be an ideal of $R$, where $I_1$ and $I_2$ are two proper ideals of $R_1$ and $R_2$, respectively. Then we have the following.
\begin{itemize}
\item [(a)] If at least one of $\Gamma''_{I_1}(R_1)$ or $\Gamma''_{I_2}(R_2)$ is not totally disconnected, Then
$g(\Gamma''_{I}(R))=3$.
\item [(b)] If $R_1\setminus I_1\not=\Bbb Z_2\setminus \{0\}$ and $R_2\setminus I_2\not=\Bbb Z_2\setminus \{0\}$, then $g(\Gamma''_{I}(R))\leq 4$.
\end{itemize}
\end{prop}
\begin{proof}
(a)
Without loss of generality,
suppose that $x, y \in V(\Gamma''_{I_2}(R_2))$ such that $x$ is adjacent to $y$.
Then by Remark \ref{1978} and Lemma \ref{2.3} (a), we have the following cycle in $\Gamma''_{I}(R)$ for some $(a, b)\in I$
$$
(a,1+b)\sim (1+a,b+x)\sim (1+a,y+b)\sim (a,1+b).
$$

(b)
Let $x \in V(\Gamma''_{I_1}(R_1))$,  $y \in V(\Gamma''_{I_2}(R_2))$ such that $x , y$ are not identity, and $(a,b) \in I$. Then the result follows from the following cycle in $\Gamma''_{I}(R)$.
$$
(x+a,b)\sim (a,1+b)\sim (1+a,b)\sim (a,y+b)\sim (x+a,b).
$$
\end{proof}

\begin{thm}\label{2.6}
Let $R=R_1\times  R_2$ be a direct product of commutative rings and $I=I_1\times I_2$ be an ideal of $R$, where $I_1$ and $I_2$ are two proper ideals of $R_1$ and $R_2$, respectively. Then the graph
 $\Gamma''_{I_1\times I_2}(R_1\times R_2)$ is connected and diam $(\Gamma''_{I_1\times I_2}(R_1\times R_2))\leq 3$.
\end{thm}
\begin{proof}
Let $(x_1, y_1)$ and  $(x_2, y_2)$ be two distinct vertices of $\Gamma''_{I_1\times I_2}(R_1\times R_2)$.
We consider the following cases:

\textbf{Case 1.}
$x_1R_1+I_1\not=R_1$ and $x_2R_1+I_1\not=R_1$. If $y_1 \in I_2$ and $y_2 \in I_2$, then $(x_1,y_1)\sim (0,1)\sim (x_2,y_2)$ is a path. If $y_1 \not\in I_2$ and $y_2 \not\in I_2$, then one can see that
$(x_1,y_1)\sim (1,0)\sim (x_2,y_2)$ is a path. Now, let $y_1\not \in I_2$ and $y_2 \in I_2$. Then we can obtain the path
$(x_1,y_1)\sim (1,0)\sim (0,1)\sim (x_2,y_2)$. The similar result holds in the case that $y_1\in I_2$ and $y_2 \not\in I_2$.

\textbf{Case 2.}
$x_1R_1+I_1\not=R_1$ and $x_2R_1+I_1=R_1$. If $y_1 \not \in I_2$, then whenever $y_2 \in I_2$, we have the path
$(x_1,y_1)\sim (x_2,y_2)$. Otherwise, $y_2 \not \in I_2$. As $(x_2,y_2) \in V(\Gamma''_{I_1\times I_2}(R_1\times R_2))$ and
$x_2R_2+I_2=R_2$, we have $y_2R_2+I_2 \not=R_2$. Now, if $(0,1) \in (x_2,y_2)(R_1\times R_2)+I_1\times I_2$, then
$x_2r_1+a_1=0$ and $r_2y_2+b_1=1$ for some $(r_1,r_2) \in R_1 \times R_2$ and $(a_1,b_1) \in I_1 \times I_2$. Thus
 $y_2R_2+I_2 =R_2$, a contradiction. Therefore, $(0,1) \not\in (x_2,y_2)(R_1\times R_2)+I_1\times I_2$. Moreover, as $x_2 \not \in I_1$, one can see that $(x_2,y_2) \not \in (0,1)(R_1\times R_2)+I_1\times I_2$. Thus $(x_1,y_1)\sim (1,0)\sim (0,1)\sim (x_2,y_2)$ is a path. Also,
 if $y_1 \in I_2$, then we have the path $(x_1,y_1)\sim (0,1)\sim (x_2,y_2)$. The similar result holds if $x_1R_1+I_1=R_1$ and $x_2R_1+I_1\not=R_1$.

\textbf{Case 3.}
$x_1R_1+I_1=R_1$ and $x_2R_1+I_1=R_1$. Then we have that $y_1R_2+I_2\not=R_2$ and $y_2R_2+I_2\not=R_2$. Thus we can apply Case 1 on the second component of ordered pairs.
Now, in view of the above cases, one can see that diam $(\Gamma''_{I_1\times I_2}(R_1\times R_2))\leq 3$.
\end{proof}

\begin{cor}\label{1407} (\cite[Theorem 4.10]{3333}).
Let $R=R_1\times  R_2$ be a direct product of commutative rings. Then the graph
 $\Gamma'(R_1\times R_2)$ is connected and diam $(\Gamma'(R_1\times R_2))\leq 3$.
\end{cor}
\begin{proof}
It is enough to set $I=0\times 0$ in Theorem \ref{2.6}.
\end{proof}

\begin{ex}
Let $R = \Bbb Z_3 \times \Bbb Z_3$ and $I =\{(0, 0)\}$.
Then $V(\Gamma''_I(R)=\{(1, 0), (2, 0), (0, 1), (0, 2)\}$ and the following figure show this graph.
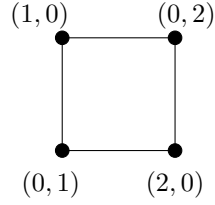
\begin{figure}[H]
				\begin{center}
				\begin{tikzpicture}	[scale=1.5]		
										\draw [fill=black] (0,1) circle (0.06);										
						\draw [fill=black] (-1,2) circle (0.06);
						\draw [fill=black] (0,2) circle (0.06);
						\draw [fill=black] (-1,1) circle (0.06);														
						\draw (0,1) -- (-0,2);
						\draw (-1,1) -- (-1,2);
	\draw (-0,2) -- (-1,2);	
			\draw (0,1) -- (-1,1);
									
						\node at (0,0.7) {$(2, 0)$};
						\node at (-1.2,2.2) {$(1,0)$};
	     				\node at (0.1,2.2) {$(0, 2)$};	
						\node at (-1.1,0.7) {$(0, 1)$};
			\end{tikzpicture}
			\end{center}
		\end{figure}
	\end{ex}
\section{Planarity of the graph $\Gamma''_I(R)$}
\begin{prop}\label{1357}
$\Gamma'(R)$ is a planar graph if and only if $\Gamma''_I(R)$ is a planar graph for each ideal $I$ of $R$.
\end{prop}
\begin{proof}
If $\Gamma''_I(R)$ is a planar graph for each ideal $I$ of $R$, then by setting $I=0$ we have $\Gamma'(R)=\Gamma''_{0}(R)$ is a planar graph. If $\Gamma'(R)$ is a planar graph, then the result follows from the fact that $\Gamma''_I(R)$ is a subgraph of $\Gamma'(R)$ for each ideal $I$ of $R$ by \cite[Corollary 2.14]{FF5726}.
\end{proof}

\begin{thm}\label{2.8}
Let $I$ be an ideal of $R$. If $\Gamma''_I(R)$ is planar, then $\vert I\vert \leq 2$ or $\vert V (\Gamma'(R/I))\vert \leq  1$.
\end{thm}
\begin{proof}
Assume contrary that  $\vert I\vert \geq 3$ and $\vert V (\Gamma'(R/I))\vert \geq  2$. Then there are
distinct adjacent vertices $x + I, y + I \in \Gamma'(R/I)$. As  $\vert I\vert \geq 3$, there are distinct
elements $0, i, j$ of $I$.
If $x$ is not adjacent to $y+i$, then we have $x \in R(y+i)+I$ or $y+i \in Rx+I$. Thus $x=ry+ri+a$ for some $r\in R$ and $a \in I$ or $y+i=tx+b$ for some $t \in R$ and $b \in I$. Therefore, $x \in Ry+I$ or $y \in Rx+I$. So we get that  $x + I$ is not adjacent to $y + I$ in  $\Gamma'(R/I)$, which is a contradiction. Thus $x$ is adjacent to $y+i$. Similarly, one can see that every vertex in $\{x, x + i, x + j\}$ is adjacent to every vertex in $\{y, y + i, y + j\}$. Therefore, the subgraph induced by $\{x, y, x + i, y +i, x + j, y + j\} = \{x, x + i, x + j\}\cup \{y, y + i, y + j\}$ is a bipartite graph.
 Note that the subgraph of $\Gamma''_I(R)$ generated by $\{x, y, x + i, y +i, x + j, y + j\} = \{x, x + i, x + j\}\cup \{y, y + i, y + j\}$ contains a subgraph isomorphic
to $K^{3,3}$. Therefore, $\Gamma''_I(R)$ is nonplanar by Kuratowski's Theorem.
\end{proof}

A \textit{pendant vertex} in a graph $G$ is a vertex $v$ that is adjacent to exactly one other vertex in $G$.

\begin{prop}\label{2.9}
Let $I$ be an ideal of $R$ such that $I\subseteq J(R)$ and $\vert V(\Gamma''_I(R)) \setminus J(R)\vert\geq 2$. Then the graph $\Gamma''_I(R)\setminus J(R)$ has no pendant vertex.
\end{prop}
\begin{proof}
Let $x \in V(\Gamma''_I(R)) \setminus J(R)$. By \cite[Theorem 2.8]{FF5726}, the graph $\Gamma''_I(R)\setminus J(R)$ is connected.
Thus there exists $y \in V(\Gamma''_I(R)) \setminus J(R)$  such that $x$ is adjacent to $y$.
If for some $i \in I$ we have $x=y+i$, then $x \in Ry+I$, which is a contradiction. Thus for each $i \in I$ we have $x\not=y+i$.
Let $0\not=a \in I$. Then $y+a \in V(\Gamma''_I(R))$. If $y+a\in J(R)$, then $y \in J(R)$ because $a \in I \subseteq J(R)$. This contradiction shows that $y+a \in V(\Gamma''_I(R))\setminus J(R)$. Clearly, $y+a$ is adjacent to $x$.
Therefore, $x$ is not a pendant vertex in $\Gamma''_I(R)\setminus J(R)$, as needed.
\end{proof}

\begin{notes}\label{1396}
Let $R$ be a finite commutative ring with identity. Then we can assume that $R$ is isomorphic to
the ring $R_1 \times \cdots \times R_n$, where $R_i$ is commutative local ring with
identity for every $i = 1, 2\ldots , n$.
\end{notes}

\begin{prop}\label{2.10}
Let $R$ be a ring as in Notation \ref{1396}, $n \geq 4$ and $I=I_1\times \cdots \times I_n$, where $I_i$ is a proper ideal of $R_i$  for every $i = 1, 2\ldots , n$. Then $\Gamma''_I(R)$ is not planar.
\end{prop}
\begin{proof}
Let $(a_1,a_2,a_3, a_4)\in I=I_1\times I_2\times I_3 \times I_4$.
The vertices of the set $\{(a_1, 1+a_2, 1+a_3, a_4), (a_1, 1+a_2, a_3, a_4), (1+a_1, 1+a_2, a_3, a_4)\}$ are adjacent to the
vertices of the set $\{(1+a_1, a_2, 1+a_3, 1+a_4), (1+a_1, a_2, a_3, 1+a_4), (1+a_1, a_2, 1+a_3, a_4)\}$, and so $K^{3,3}$ is a subgraph of $\Gamma''_I(R)$. Hence  $\Gamma''_I(R)$ is not planar by Kuratowski's Theorem.
\end{proof}

\begin{prop}\label{2.11}
Let $R$ be a ring as in Notation \ref{1396}, $n = 3$, and $I=I_1\times I_2\times I_3$, where $I_i$ is a proper ideal of $R_i$  for every $i = 1, 2,3$ . If there exists $1 \leq i \leq 3$, such that $R_i$ has at least three elements, then $\Gamma''_I(R)$ is not planar.
\end{prop}
\begin{proof}
Let $(a_1,a_2,a_3)\in I_1\times I_2\times I_3$.
 Without loss of generality, suppose that $R_1$ has at least three elements.
Then there exists a subdivision of $K^{3,3}$ in $\Gamma''_I(R)$ as in following figure, where $x \in  R_1\setminus \{0, 1\}$.
So, by Kuratowski’s Theorem, $\Gamma''_I(R)$ is not planar.
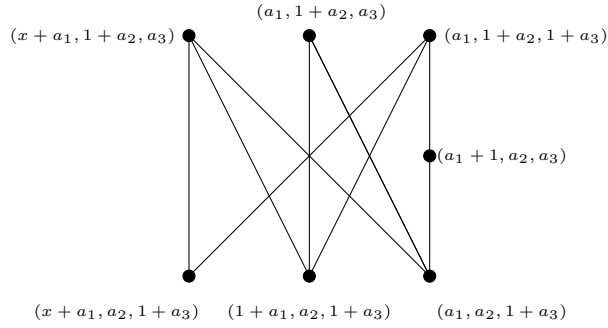
\begin{figure}[H]
						\begin{center}
				\begin{tikzpicture}	[scale=1.6]		
					\begin{scope}[shift={(-1,0)}]
						\draw [fill=black] (0,0) circle (0.05);										
						\draw [fill=black] (1,1) circle (0.05);
						\draw [fill=black] (-1,2) circle (0.05);
						\draw [fill=black] (0,2) circle (0.05);
						\draw [fill=black] (1,2) circle (0.05);
						\draw [fill=black] (1,0) circle (0.05);
						\draw [fill=black] (-1,0) circle (0.05);
									
						\draw (1,1) -- (1,0) ;
						\draw (0,0) -- (1,2) ;
						\draw (1,2) -- (1,1) ;
						\draw (1,2) -- (-1,0) ;
						\draw (0,0) -- (-1,2) ;
						\draw (0,2) -- (1,0) ;
						\draw (1,0) -- (0,2) ;
						\draw (0,0) -- (-0,2) ;
						\draw (-1,0) -- (-1,2) ;
						\draw (-1,2) -- (1,0) ;
					
						\node at (0,-0.3) {\tiny $(1+a_1, a_2,1+a_3 )$};
						\node at (-1.8,2) {\tiny $(x+a_1,1+a_2,a_3)$};
						\node at (1.6,1) {\tiny $(a_1+1, a_2,a_3)$};
						\node at (1.8,2) {\tiny $(a_1, 1+a_2,1+a_3)$};
						\node at (0.1,2.2) {\tiny $(a_1,1+a_2,a_3)$};	
						\node at (1.6,-0.3) {\tiny $(a_1, a_2,1+a_3)$};
						\node at (-1.6,-0.3) {\tiny $(x+a_1, a_2,1+a_3)$};
					\end{scope}						
				\end{tikzpicture}
						\end{center}
		\end{figure}
\end{proof}

\begin{cor}\label{2}
Let $R$ be a ring as in Notation \ref{1396}, $I$ be an ideal of $R$, and $n=3$. Then $\Gamma''_I(R)$ is a planar graph if and only if $R\cong \Bbb Z_2 \times \Bbb Z_2 \times\Bbb Z_2$.
\end{cor}
\begin{proof}
This follows from Proposition \ref{1357},  \cite[Corollary 2.3]{2222}, and Proposition \ref{2.11}.
\end{proof}

Let $R$ be a ring as in Notation \ref{1396}. In the following, we assume that $n = 2$ and $I=I_1 \times I_2$, where $I_1$ is a proper ideal of $R_1$ and $I_2$ is a proper ideal of $R_2$. If $\vert R_1\vert\geq 4$ and $\vert R_2\vert\geq 4$,
then the vertices of the set $\{(a_1+1, a_2), (a_1+x_1, a_2), (a_1+y_1, a_2)\}$ are adjacent to the vertices of the
set $\{(a_1, 1+a_2), (a_1, x_2+a_2), (a_1, y_2+a_2)\}$, where $x_i, y_i \in R_i\setminus \{0, 1\}$, for $i = 1, 2$ and $(a_1, a_2) \in I$. Thus $K^{3,3}$ is a subgraph of $\Gamma''_I(R)$. Therefore,  $\Gamma''_I(R)$ is not planar.

Now assume that $\vert R_1\vert \leq 3$ or $\vert R_2\vert\leq 3$. Firstly, we state the following lemma.

\begin{lem}\label{2.13}
Let $R$ be a ring as in Notation \ref{1396}, $n = 2$,  and $I=I_1 \times I_2$, where $I_1$ is a proper ideal of $R_1$ and $I_2$ is a proper ideal of $R_2$.  If $V(\Gamma''_{I_1}(R_1))$ or $V(\Gamma''_{I_2}(R_2))$ has more than one element, then $\Gamma''_I(R)$ is not planar.
\end{lem}
\begin{proof}
Without loss of generality, we may assume that $V(\Gamma''_{I_2}(R_2))$ has more than one element. Let $x_1, x_2 \in V(\Gamma''_{I_2}(R_2))$ and $(a_1, a_2) \in I$. Set $u_1 := 1 + x_1$ and $u_2 := 1 + x_2$. Then, since $R_2$ is local, we have that
$u_1, u_2 \in U(R_2)$. Now the vertices of the set $\{(a_1+1, a_2), (a_1+1, a_2+x_1), (a_1+1, a_2+x_2)\}$ are adjacent to
the vertices of the set $\{(a_1, a_2+u_1), (a_1, a_2+u_2), (a_1, 1+a_2)\}$, and so $K^{3,3}$ is a subgraph of $\Gamma''_I(R)$.
Therefore $\Gamma''_I(R)$ is not planar by Kuratowski's Theorem.
\end{proof}

Now, we have the following theorem which characterize all finite non-local rings
with the graph $\Gamma''_I(R)$ is planar.
\begin{thm}\label{2.14}
Let $R$ be a finite non-local ring with $n$ elements and $I=I_1\times I_2\times \cdots \times I_n$ be an ideal of $R$, where $I_i$ is a proper ideal of $R_i$ for $i=1,2, \ldots, n$. Then $\Gamma''_I(R)$ is planar if and only if $R$ is
one of the following rings:
$$
\Bbb Z_2 \times  \Bbb Z_2 \times  \Bbb Z_2,
$$
$$
\Bbb Z_2 \times  F, \   \   \  \Bbb Z_2 \times  \Bbb Z_4, \   \   \ \Bbb Z_2 \times \frac{\Bbb Z_2[X]}{(X^2)\Bbb Z_2[X]},
$$
$$
\Bbb Z_3 \times  F, \   \   \ \Bbb Z_3 \times  \Bbb Z_4, \   \   \ \Bbb Z_3 \times \frac{\Bbb Z_2[X]}{(X^2)\Bbb Z_2[X]},
$$
where $F$ is a finite field.
\end{thm}
\begin{proof}
This follows from \cite[Theorem 2.5]{2222}, Proposition \ref{1357}, and Lemma \ref{2.13}.
\end{proof}

\begin{thm}\label{2.15}
Let $(R,\mathfrak{m})$ be a finite local ring such that $\mathfrak{m}$ is a principal ideal of $R$ and let $I$ be an ideal of $R$.
Then $\Gamma''_I(R)$ is planar.
\end{thm}
\begin{proof}
This follows from \cite[Theorem 2.6]{2222} and Proposition \ref{1357}.
\end{proof}

For an ideal $I$ of a finite commutative ring $R$ with identity, ara$(I)$ is the smallest of the cardinalities of minimal
generating sets of $I$.
\begin{prop}\label{2.16}
Let $(R,\mathfrak{m})$ be a finite local ring such that $\mathfrak{m}$ is not a principal ideal of $R$ and let $I$ be an ideal of $R$.
If $\Gamma''_I(R)$ is planar, then ara$(\mathfrak{m}/I) \leq 4$.
\end{prop}
\begin{proof}
Assume to the contrary that each minimal generating set of $\mathfrak{m}/I$ has more
than four elements, and that $x_1+I, \dots , x_5 +I$ belong to a minimal generating set of $\mathfrak{m}/I$.
Then a subgraph of $\Gamma''_I(R)$ with vertex-set $\{a+x_1, a+x_2, a+x_3,  a+x_4,  a+x_5\}$ is isomorphic to $K^5$ for each $a \in I$. Thus by Kuratowski's Theorem,  $\Gamma''_I(R)$ is not planar, which is a contradiction.
\end{proof}

An undirected graph is said to be an \textit{outerplanar graph} if it can be drawn in the plane
without crossings in such a way that all of the vertices belong to the unbounded face
of the drawing. There is a characterization for outerplanar graphs that says a graph
is outerplanar if and only if it does not contain a subdivision of $K^4$ or $K^{2,3}$.
For more information see \cite{1979}.

\begin{prop}\label{1359}
$\Gamma'(R)$ is an outerplanar graph if and only if $\Gamma''_I(R)$ is an outerplanar graph for each ideal $I$ of $R$.
\end{prop}
\begin{proof}
This is similar to the proof of Proposition \ref{1357}.
\end{proof}

\begin{prop}\label{1360}
Let $(R,\mathfrak{m})$ be a finite local ring such that $\mathfrak{m}/I$ is not a principal ideal of $R$, where $I$ is an ideal of $R$.
If in a minimal generating set of $\mathfrak{m}/I$ there exist distinct elements $x+I$ and $y+I$ with
$\vert U(R)x\vert \geq 3$, $\vert U(R)y\vert \geq 3$, then $\Gamma''_I(R)$ is not planar.
\end{prop}
\begin{proof}
Suppose that in a minimal generating set of $\mathfrak{m}/I$ there exist distinct elements $x+I$ and $y+I$ with
$\vert U(R)x\vert \geq 3$, $\vert U(R)y\vert \geq 3$. Then there are distinct elements $u_1x, u_2x, u_3x\in U(R)x$ and $v_1y, v_2y, v_3y\in U(R)y$. Now, it is easy to see that the vertices of the
set $\{u_1x, u_2x, u_3x\}$ are adjacent to the vertices of the set $\{u_1x, u_2x, u_3x\}$. Thus $K^{3,3}$
is a subgraph of $\Gamma''_I(R)$. Therefore, $\Gamma''_I(R)$ is not planar by Kuratowski's Theorem.
\end{proof}

\begin{thm}\label{1000}
Let $R = R_1 \times R_2$ such that $R_1$ and $R_2$ are arbitrary
local rings and $I=I_1\times I_2$, where $I_1\not=R_1$ is an ideal of $R_1$ and $I_2\not=R_2$ is an ideal of $R_2$.
If there exist elements $x$ and $y$ in $V(\Gamma''_{I_1}(R_1))$ (or $V(\Gamma''_{I_2}(R_2))$) such that $x$ is
adjacent to $y$, then $\Gamma''_I(R)$ is not planar.
\end{thm}
\begin{proof}
Without loss of generality, suppose that $x,y \in V(\Gamma''_{I_2}(R_2))$ and $x$ is adjacent
to $y$. Set $z := 1 + x$. Clearly $z\in U(R_2)$. Let $a_1\in I_1\not =R_1$ and $a_2 \in I_2\not=R_2$.
Now, we have the following subdivision of
$K^{3,3}$ in $\Gamma''_I(R)$. Therefore, $\Gamma''_I(R)$ is not planar by Kuratowski's Theorem.
 
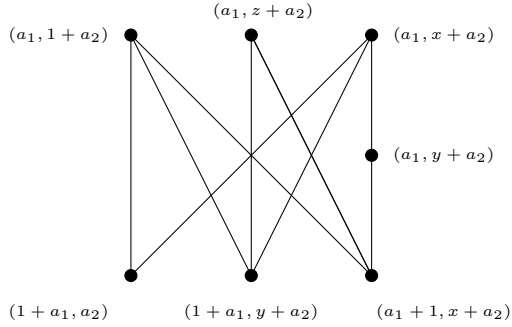
\begin{figure}[H]
						\begin{center}
				\begin{tikzpicture}	[scale=1.6]		
					\begin{scope}[shift={(-1,0)}]
						\draw [fill=black] (0,0) circle (0.05);										
						\draw [fill=black] (1,1) circle (0.05);
						\draw [fill=black] (-1,2) circle (0.05);
						\draw [fill=black] (0,2) circle (0.05);
						\draw [fill=black] (1,2) circle (0.05);
						\draw [fill=black] (1,0) circle (0.05);
						\draw [fill=black] (-1,0) circle (0.05);
									
						\draw (1,1) -- (1,0) ;
						\draw (0,0) -- (1,2) ;
						\draw (1,2) -- (1,1) ;
						\draw (1,2) -- (-1,0) ;
						\draw (0,0) -- (-1,2) ;
						\draw (0,2) -- (1,0) ;
						\draw (1,0) -- (0,2) ;
						\draw (0,0) -- (-0,2) ;
						\draw (-1,0) -- (-1,2) ;
						\draw (-1,2) -- (1,0) ;
					
						\node at (0,-0.3) {\tiny $(1+a_1, y+a_2)$};
						\node at (-1.6,2) {\tiny $(a_1,1+a_2)$};
						\node at (1.6,1) {\tiny $(a_1, y+a_2)$};
						\node at (1.6,2) {\tiny $(a_1, x+a_2)$};
						\node at (0.1,2.2) {\tiny $(a_1,z+a_2)$};	
						\node at (1.6,-0.3) {\tiny $(a_1+1, x+a_2)$};
						\node at (-1.6,-0.3) {\tiny $(1+a_1, a_2)$};
					\end{scope}						
				\end{tikzpicture}
						\end{center}
		\end{figure}
	\end{proof}
	
{\bf Acknowledgement.}
The author would like to thank the referee for his/her helpful comments.

\bibliographystyle{amsplain}

\begin{thebibliography}{10}
\bibitem{1}
M.~Afkhami and K.~Khashyarmanesh, \emph{The cozero-divisor graph of a commutative ring}, Southeast Asian Bull. Math., \textbf{35} (2011), 753--762.

\bibitem{1111}
M.~Afkhami and K.~Khashyarmanesh, \emph{On the Cozero-Divisor Graphs of Commutative Rings and Their Complements}, Bull. Malays. Math. Sci. Soc. (2) \textbf{35} (4) (2012), 935–944.

\bibitem{2222}
M.~Afkhami and K.~Khashyarmanesh, \emph{Planar, outerplanar, and ring graph of the cozero-divisor graph of a finite commutative ring}, J. Algebra Appl.,  \textbf{11} (06) (2012), 1250103 (9 pages).

\bibitem{3333}
M. Afkhami and K. Khashyarmanesh, \emph{On the Cozero-Divisor Graphs of Commutative Rings}, Applied Mathematics, \textbf{4} (2013) 979--985.

\bibitem{2}
D.F.~Anderson and P.S.~Livingston, \emph{The zero-divisor graph of a commutative ring}, J. Algebra, \textbf{217} (1999), 434--447.

\bibitem{00}
H. Ansari-Toroghy, F. Farshadifar, and F. Mahboobi-Abkenar, \emph{An ideal-based cozero-divisor graph of a commutative ring}, Bol. Soc. Parana. Mat., \textbf{40} (3) (2022), 1-8.

\bibitem{BM76}
J. A Bondy and U.S.R. Murty, \emph{Graph Theory with Applications}, American Elsevier, New York, 1976.

\bibitem{FF5726}
F. Farshadifar, \emph{A generalization of the cozero-divisor graph of a commutative ring}, Discrete Mathematics, Algorithms and Applications, \textbf{17} (5), (2025), 2450073.

\bibitem{3}
S. P.~Redmond, \emph{An ideal-based zero-divisor graph of a commutative ring}, Comm. Algebra, \textbf{31} (2003), 4425-4443.

\bibitem{1979}
M. M. Sys\l o,  \emph{Characterizations of outerplanar graphs}. Discrete Math. 26 (1), (1979), 47--53.
\end{thebibliography}

\end{document}